\documentclass[11pt]{article}

\usepackage{amsmath,amsthm,amssymb,amsxtra,amsfonts,amsbsy,mathrsfs}
\usepackage[all,2cell]{xy} \UseAllTwocells
\SilentMatrices
\newtheorem{theorem}[subsection]{Theorem}

\newtheorem{corollary}[subsection]{Corollary}

\newtheorem{lemma}[subsection]{Lemma}
\newtheorem{proposition}[subsection]{Proposition}

\newtheorem{sublemma}[subsubsection]{Lemma}
\theoremstyle{definition}
\newtheorem{notation-convention}[subsection]{Notations
and Conventions}
\newtheorem{definition}[subsection]{Definition}

\newtheorem{remark}[subsection]{Remark}

\newtheorem{anitem}[subsubsection]{}
\newtheorem{blank}[subsection]{}

\newcommand{\leftexp}[2]{{\vphantom{#2}}^{#1}{#2}}
\newcommand{\bb}{\mathbb}
\newcommand{\s}{\mathscr}

\begin{document}

\title{Decomposition Theorem for Perverse sheaves on Artin
stacks over finite fields}
\author{Shenghao Sun}
\date{}
\maketitle

\begin{abstract}
We generalize the decomposition theorem for perverse sheaves
to Artin stacks with affine stabilizers over finite
fields.
\end{abstract}

\section{Introduction}\label{sec-counter-ex}

Among the most important theorems on the topology of (families of)
projective complex manifolds, there are the hard Lefschetz theorem,
degeneration of the Leray spectral sequence at $E_2,$ and Deligne's
semisimplicity theorem of cohomology local systems. They all fail
for algebraic varieties with singularities, but if we replace the
ordinary cohomology groups with the intersection cohomology groups,
introduced by Goresky and Mac Pherson, these results turn out to
hold, and what lies in the center of the story is the notion of
\textit{perverse sheaves} and the \textit{decomposition theorem}
for them, first proved by Beilinson, Bernstein, Deligne and Gabber
\cite{BBD}. For more discussion on these, see \cite{deC-M}.

Then the notion of perverse sheaves was generalized to spaces with
group actions (the so-called \textit{equivariant perverse sheaves}),
and then more generally, to algebraic stacks \cite{LO3}. The key
observation here is that perverse sheaves can be glued with respect
to the \textit{smooth topology.} It is interesting to know if the
decomposition theorem generalizes. The case for equivariant perverse
sheaves has been proved in (\cite{BL}, 5.3). Let us remark here that
for algebraic stacks, it does not follow directly from the case for
algebraic varieties via the proper base change, because the
decomposability of a complex of sheaves is \textit{not} local for
the smooth topology, as the following counter-example of Drinfeld shows.

Let $E$ be a complex elliptic curve, and let
$f:\text{pt = Spec }\bb C\to BE$ be the natural
projection; this is a representable proper smooth map. A perverse
sheaf on $BE$ is the same as a lisse sheaf (which turns out to be
constant), appropriately shifted. There is a
natural non-zero morphism $\underline{\bb C}_{BE}\to
Rf_*\underline{\bb C}_{\text{pt}},$ adjoint to the
isomorphism $f^*\underline{\bb C}_{BE}\simeq
\underline{\bb C}_{\text{pt}},$ but there is no
non-zero morphism in the other direction, because
$$
Hom(Rf_*\underline{\bb C}_{\text{pt}},\underline{\bb
C}_{BE})=Hom(\underline{\bb C}_{\text{pt}},f^!\underline{\bb
C}_{BE})=Hom(\underline{\bb C}_{\text{pt}},
\underline{\bb C}_{\text{pt}}[2])=0.
$$
Here the $Hom$'s are taken in the derived categories.
Similarly, the non-zero natural map
$Rf_*\underline{\bb C}_{\text{pt}}\to R^2f_*\underline
{\bb C}_{\text{pt}}[-2]=\underline{\bb C}_{BE}[-2]$ lies in
$$
Hom(Rf_*\underline{\bb C}_{\text{pt}},\underline{\bb
C}_{BE}[-2])=Hom(\underline{\bb C}_{\text{pt}},f^!\underline{\bb
C}_{BE}[-2])=Hom(\underline{\bb C}_{\text{pt}},\underline{\bb
C}_{\text{pt}})=\bb C,
$$
but the Hom set in the other direction is zero:
$$
Hom(\underline{\bb C}_{BE}[-2],Rf_*\underline{\bb
C}_{\text{pt}})=Hom(f^*\underline{\bb C}_{BE}[-2],
\underline{\bb C}_{\text{pt}})=Hom(\underline{\bb
C}_{\text{pt}}[-2],\underline{\bb C}_{\text{pt}})=0.
$$
Therefore, $Rf_*\underline{\bb C}$ is not semi-simple
(since it is not a direct sum of the
$(R^if_*\underline{\bb C})[-i]$'s). The
same argument applies to finite fields, with $\underline{\bb
C}$ replaced by $\overline{\bb Q}_{\ell}.$

\begin{remark}\label{source}
This example was first given by Drinfeld, who asked for
the reason of the failure of the usual argument for
schemes. Later, it was communicated by J. Bernstein to
Y. Varshavsky, who asked M. Olsson in an email
correspondence. Olsson kindly shared this email with me,
and explained to me that the reason is the failure of the
upper bound of weights in \cite{Del2} for $BE.$
\end{remark}

In the following we explain why the usual proof (as in
\cite{BBD}) fails for $f.$ The proof of the
decomposition theorem over $\bb C$ relies on the
decomposition theorem over finite fields (\textit{loc. cit.},
5.3.8, 5.4.5), so it suffices to explain why the proof of
(\textit{loc. cit.}, 5.4.5) fails for $f:\text{Spec }\bb F_q
\to BE,$ for an $\bb F_q$-elliptic curve $E.$

Let $K_0=Rf_*\overline{\bb Q}_{\ell}.$ The perverse $t$-structure
agrees with the standard $t$-structure on $\text{Spec }\bb F_q,$ and
by definition (\cite{LO3}, 4), we have $\leftexp{p}{\s H}^iK_0=\s
H^{i+1}(K_0)[-1]$ on $BE,$ and so
$$
\bigoplus_i(\leftexp{p}{\s H}^iK)[-i]=\bigoplus_i
(\s H^iK)[-i].
$$
Each $R^if_*\overline{\bb Q}_{\ell}[-i]$ is pure of
weight 0. In the proof of (\cite{BBD}, 5.4.5), the exact
triangles
$$
\xymatrix@C=.5cm{
\tau_{<i}K_0 \ar[r] & \tau_{\le i}K_0 \ar[r] & (\s
H^iK_0)[-i] \ar[r] &}
$$
would split geometrically, because $Ext^1((\s H^iK)[-i],
\tau_{<i}K)$ would have weights $>0.$ We will see that this
group is pure of weight 0, and in fact has 1 as a Frobenius eigenvalue.
For simplicity, we denote $H^i(\s X,\overline{\bb
Q}_{\ell})$ by $H^i(\s X).$

Let $\pi:BE\to\text{Spec }\bb F_q$ be the structural
map; then $\pi\circ f=\text{id}.$ Since $E$ is connected,
the sheaf $R^if_*\overline{\bb Q}_{\ell}$ is the
inverse image of some sheaf on $\text{Spec }\bb
F_q,$ namely $f^*R^if_*\overline{\bb Q}_{\ell}.$ By
smooth base change, it is isomorphic to $\pi^*H^i(E)$
as a $\text{Gal}(\overline{\bb F}_q/\bb F_q)$-module. In particular,
$R^0f_*\overline{\bb Q}_{\ell}=\overline{\bb
Q}_{\ell},\ R^1f_*\overline{\bb Q}_{\ell}\cong\pi^*
H^1(E)$ and $R^2f_*\overline{\bb Q}_{\ell}=\overline{
\bb Q}_{\ell}(-1).$ Then the exact triangle above becomes
\begin{gather*}
i=2:\qquad\xymatrix@C=.5cm{
\tau_{\le1}K_0 \ar[r] & K_0 \ar[r] & \overline{\bb
Q}_{\ell}(-1)[-2] \ar[r] &} \\
i=1:\qquad\xymatrix@C=.5cm{
\overline{\bb Q}_{\ell} \ar[r] & \tau_{\le1}K_0
\ar[r] & \pi^*H^1(E)[-1] \ar[r] &.}
\end{gather*}
Apply $Ext^*(\overline{\bb Q}_{\ell}(-1)[-2],-)$ to
the second triangle. One can compute $H^*(BE)$ by a
theorem of Borel (see (\cite{Sun}, 7.2)): $H^{2i-1}(BE)=0,$
and $H^{2i}(BE)=\text{Sym}^iH^1(E).$ Let $\alpha$ and
$\beta$ be the eigenvalues of the Frobenius $F$ on
$H^1(E).$ We have
$$
Ext^1(\overline{\bb Q}_{\ell}(-1)[-2],
\overline{\bb Q}_{\ell})=Ext^3(\overline{\bb
Q}_{\ell},\overline{\bb Q}_{\ell}(1))=H^3(BE)(1)=0,
$$
and
\begin{equation*}
\begin{split}
Ext^1(\overline{\bb Q}_{\ell}(-1)[-2],\pi^*H^1(E)[-1])
&=H^2(BE)\otimes H^1(E)(1) \\
&=H^1(E)\otimes H^1(E)(1)=End(H^1(E)),
\end{split}
\end{equation*}
which is 4-dimensional with eigenvalues
$\alpha/\beta,\beta/\alpha,1,1,$ and
$$
Ext^2(\overline{\bb Q}_{\ell}(-1)[-2],
\overline{\bb Q}_{\ell})=H^4(BE)(1),
$$
which is 3-dimensional with eigenvalues $\alpha/\beta,
\beta/\alpha,1.$ This implies that the kernel
\begin{equation*}
\begin{split}
Ext^1(\overline{\bb Q}_{\ell}(-1)[-2],\tau_{\le1}K)
&= \\
&\text{Ker}\Big(Ext^1(\overline{\bb Q}_{\ell}(-1)[-2],
\pi^*H^1(E)[-1])\to Ext^2(\overline{\bb
Q}_{\ell}(-1)[-2],\overline{\bb Q}_{\ell})\Big)
\end{split}
\end{equation*}
is non-zero, pure of weight 0, and has 1 as a Frobenius
eigenvalue. So the first exact triangle above does not
necessarily (in fact does not, as the argument in the
beginning shows) split geometrically. Also
$$
Ext^1(\pi^*H^1(E)[-1],\overline{\bb Q}_{\ell})=
Ext^2(\overline{\bb Q}_{\ell},\pi^*H^1(E)^{\vee})=
H^1(E)\otimes H^1(E)^{\vee}=End(H^1(E))
$$
is 4-dimensional and has eigenvalues
$\alpha/\beta,\beta/\alpha,1,1,$ hence the proof for the
geometric splitting of the second exact triangle fails
too.

\vskip.5truecm

In \cite{LO3}, Laszlo and Olsson generalized the theory of
perverse sheaves to Artin stacks locally of finite type
over some field. In \cite{Sun}, we proved that for Artin
stacks of finite type over a finite field, with affine
stabilizers (\ref{affine-stab}),
Deligne's upper bound of weights for the compactly
supported cohomology groups still applies. In this paper,
we will show that for such stacks, similar argument as in
\cite{BBD} gives the decomposition theorem. We state it as follows
(see (\ref{torsion}) for the notation).

\begin{theorem}\label{main-thm}
Let $f:\s X_0\to\s Y_0$ be a proper morphism of finite diagonal
between $\bb F_q$-algebraic stacks of finite type with affine
stabilizers (\ref{affine-stab}), and let $K_0$ be an $\iota$-pure
$\overline{\bb Q}_{\ell}$-complex on $\s X_0.$ Then we have
$$
Rf_*K\simeq\bigoplus_{i\in\bb Z}(\leftexp{p}{R}^if_*K)[-i],
$$
and each $\leftexp{p}{R}^if_*K$ is a semi-simple perverse sheaf on
$\s Y.$ Consequently, if $K_0$ is a semi-simple $\overline{\bb
Q}_{\ell}$-perverse sheaf on $\s X_0,$ the conclusion above also holds.
\end{theorem}

\textbf{Organization.} In $\S\ref{sec-prototype}$ we complete the
proof of the structure theorem for $\iota$-mixed sheaves on stacks,
as claimed in (\cite{Sun}, 2.7.1). In $\S\ref{sec-decomp-F_q},$ we
generalize the decomposition theorem for perverse sheaves on stacks
over finite fields, using weight theory. In the end, we mention the
decomposition theorem for stacks over $\bb C.$

\begin{notation-convention}\label{notat-conv}

\begin{anitem}\label{frob}
We fix an algebraic closure $\bb F$ of the finite
field $\bb F_q$ with $q$ elements. Let $F$ or $F_q$ be
the $q$-geometric Frobenius, namely the $q$-th root
automorphism on $\bb F.$ Let $\ell$ be a prime number,
$\ell\ne p,$ and fix an embedding of fields
$\overline{\bb Q}_{\ell}\overset{\iota}{\to}\bb
C.$ For $z\in\bb C,$ let $w_q(z)=2\log_q|z|.$
\end{anitem}

\begin{anitem}\label{ft}
For the definition of an Artin stack (or an algebraic stack), we
refer to (\cite{Ols2}, 1.2.22). We only consider algebraic stacks
\textit{of finite type} over the base.
\end{anitem}

\begin{anitem}\label{torsion}
Objects over $\bb F_q$ will be denoted with a subscript
$_0,$ and suppression of it means passing to $\bb F$ by extension
of scalars. For instance, if $K_0$ is a $\overline{\bb
Q}_{\ell}$-complex of sheaves on an $\bb F_q$-Artin stack $\s X_0,$
then $K$ denotes its inverse image
on $\s X:=\s X_0\otimes_{\bb F_q}\bb F.$ For $b\in\overline
{\bb Q}_{\ell}^*,$ let $\overline{\bb Q}_{\ell}^{(b)}$
be the lisse Weil sheaf of rank one on $\text{Spec }\bb
F_q$ corresponding to the character that sends $F_q$ to
$b$ (see (\cite{Sun}, 2.4)).
\end{anitem}

\begin{anitem}
For an algebraic stack $X$ over a field $k,$ we say it is
\textit{essentially smooth} if $(X_{\overline{k}})
_{\text{red}}$ is smooth over $\overline{k}.$
\end{anitem}

\begin{anitem}
For a map $f:X\to Y$ and a complex of sheaves $K$ on $Y,$ we
sometimes write $H^n(X,K)$ for $H^n(X,f^*K).$
\end{anitem}

\begin{anitem}
We will denote $Rf_*,Rf_!,Lf^*$ and $Rf^!$ by $f_*,f_!,f^*$ and
$f^!$ respectively. We use $Hom$ and $Ext$ (resp. $\s Hom$ and $\s
Ext$) for the global Hom and Ext (resp. sheaf Hom and Ext).
\end{anitem}

\begin{anitem}
We will only consider the middle perversity (\cite{BBD}, 4.0). We use
$\leftexp{p}{\s H}^i$
and $\leftexp{p}{\tau}_{\le i}$ to denote cohomology and truncations with
respect to this perverse $t$-structure.
\end{anitem}
\end{notation-convention}

\begin{flushleft}
\textbf{Acknowledgment.}
\end{flushleft}

I would like to thank my advisor Martin Olsson for
introducing this topic to me, and giving so many suggestions
during the writing. Yves Laszlo and Weizhe Zheng pointed out
some mistakes and gave many helpful comments. Many
people, especially Brian Conrad and Matthew Emerton, have
helped to answer my questions related to this paper on
mathoverflow. The revision of the paper was done during the stay in Ecole
polytechnique CMLS (UMR 7640) and Universit\'e Paris-Sud (UMR 8628), while
I was supported by ANR grant G-FIB.

\section{The prototype: the structure theorem of $\iota$-mixed
sheaves on stacks}\label{sec-prototype}

We generalize the structure theorem of $\iota$-mixed
sheaves (\cite{Del2}, 3.4.1) to stacks. This result has little
to do with the rest of this paper (except in (\ref{5.3.4})), but it is the
prototype of the corresponding
results (e.g. weight filtrations and the decomposition
theorem) for perverse sheaves. In this section, sheaves are
understood as Weil sheaves. See (\cite{Sun}, 2.4.3) for the definitions
of punctually $\iota$-pure sheaves and $\iota$-mixed sheaves.

\begin{theorem}\label{3.4.1}\emph{(stack version of
(\cite{Del2}, 3.4.1))}
Let $\s X_0$ be an $\bb F_q$-algebraic stack.

(i) Every $\iota$-mixed sheaf $\s F_0$ on $\s X_0$ has a unique
decomposition $\s F_0=\bigoplus_{b\in\bb{R/Z}}\s F_0(b),$
called the \emph{decomposition according to the weights
mod }$\bb Z,$ such that the punctual $\iota$-weights
of $\s F_0(b)$ are all in the coset $b.$ This
decomposition, in which almost all the $\s F_0(b)$'s
are zero, is functorial in $\s F_0.$

(ii) Every $\iota$-mixed lisse sheaf $\s F_0$ with
integer punctual $\iota$-weights on $\s X_0$ has a
unique finite increasing filtration $W$ by lisse
subsheaves, called the \emph{weight filtration}, such that
$\emph{Gr}_i^W$ is punctually $\iota$-pure of weight $i.$
Every morphism between such sheaves on $\s X_0$ is
strictly compatible with their weight filtrations.

(iii) If $\s X_0$ is a normal algebraic stack (i.e. it has
a normal presentation), and
$\s F_0$ is a lisse and punctually $\iota$-pure
sheaf on $\s X_0,$ then $\s F$ on $\s X$ is semi-simple.
\end{theorem}

\begin{proof}
(i) and (ii) are proved in (\cite{Sun}, 2.7.1), where
(iii) is claimed without giving a detailed proof.
Here we complete the proof of (iii).

First of all, note that we may make a finite extension of
the base field $\bb F_{q^v}/\bb F_q.$
From the proof of (\cite{LO3}, 8.3), we see that if
$\s U\subset\s X$ is an open substack, and
$\s G_{\s U}$ is a lisse subsheaf of $\s
F|_{\s U},$ then it extends to a unique lisse subsheaf
$\s G\subset\s F.$ Applying the full-faithfulness
in (\textit{loc. cit.}) we see that if $\s F|_{\s U}$
is semi-simple, so also is $\s F.$ Therefore, we may shrink
$\s X$ to a dense open substack $\s U,$ and
replace $\s X_0$ by some model of $\s U$ over a
finite extension $\bb F_{q^v}.$ We can then assume
$\s X_0$ is smooth and geometrically connected.

Following the proof (\cite{Del2}, 3.4.5), it suffices to
show (\cite{Del2}, 3.4.3) for stacks. We claim that, if
$\s F_0$ is lisse and punctually $\iota$-pure of
weight $w,$ then $H^1(\s{X,F})$ is $\iota$-mixed of
weights $\ge1+w.$ The conclusion follows from this claim.

Let $N=\dim\s X_0.$ By Poincar\'e duality, it
suffices to show that, for every lisse sheaf $\s
F_0,$ punctually $\iota$-pure of weight $w,\ H^{2N-1}_c
(\s{X,F})$ is $\iota$-mixed of weights $\le2N-1+w.$
To show this, we may shrink $\s X_0$ to assume that
the inertia $\s I_0\to\s X_0$ is flat, with rigidification
$\pi:\s X_0\to X_0$ (cf. (\cite{Ols2}, 1.5)). We have the spectral
sequence
$$
H^r_c(X,R^k\pi_!\s F)\Longrightarrow
H^{r+k}_c(\s{X,F}),
$$
so let $r+k=2N-1.$ Note that $k$ can only be of the form
$-2i-2d,$ for $i\ge0,$ where $d=\text{rel. dim}(\s
I_0/\s X_0).$ So we have $r=2\dim X_0+2i-1,$ and in
order for $H^r_c(X,-)$ to be non-zero, $i=0.$ Then
$$
H^{2N-1}_c(\s{X,F})=H^{2\dim X-1}_c(X,R^{-2d}\pi_!\s F).
$$
The claim follows from the fact that $R^{-2d}\pi_!\s
F_0$ is punctually $\iota$-pure of weight $w-2d.$ To see this
fact, it suffices to show that $H^{-2d}_c(BG,\s F)$ has
weight $w-2d,$ for any $\bb F_q$-algebraic group $G_0$ of
dimension $d,$ and any lisse punctually $\iota$-pure sheaf
$\s F_0$ on $BG_0$ of weight $w.$ By considering the Leray spectral
sequence for the natural map $BG_0\to B\pi_0(G_0),$ we reduce to
the case where $G_0$ is connected, and this case is clear. See
the proof of (\cite{Sun}, 1.4) for more details.
\end{proof}

\section{Decomposition theorem for stacks over $\bb
F_q$}\label{sec-decomp-F_q}

For an $\bb F_q$-algebraic stack $\s X_0,$ let $D_m(\s
X_0,\overline{\bb Q}_{\ell})$ be the full subcategory of mixed
complexes in $D_c(\s X_0,\overline{\bb Q}_{\ell})$ (see (\cite{LO3},
9.1)). It is stable under the six operations (\cite{Sun}, 2.11,
2.12) and the perverse truncations $\leftexp{p}{\tau}_{\le0}$ and
$\leftexp{p}{\tau}_{\ge0}.$ The latter can be checked smooth
locally, and hence follows from (\cite{BBD}, 5.1.6). The core of
$D_m (\s X_0,\overline{\bb Q}_{\ell})$ with respect to this induced
perverse $t$-structure is called the category of \textit{mixed
perverse sheaves} on $\s X_0,$ as defined in (\cite{LO3}, 9.1). This
is a Serre subcategory of all $\overline{\bb Q}_{\ell}$-perverse
sheaves $\text{Perv}(\s X_0),$ i.e. it is closed under sub-quotients
and extensions. For sub-quotients, see (\cite{LO3}, 9.3). For
extensions, note that a short exact sequence of perverse sheaves is
an exact triangle in $D_c^b,$ and we may apply (the mixed variant
of) (\cite{Sun}, 2.5 iii).

Nevertheless, in this paper, we will consider the more general notion of
\textit{$\iota$-mixed complexes} and in particular,
\textit{$\iota$-mixed perverse sheaves}. This weaker
condition will be sufficient for the purpose of proving the
decomposition theorem.
In fact, Lafforgue has proved the conjecture of Deligne that,
all (Weil) sheaves are $\iota$-mixed, for any $\iota.$ See (\cite{Lau},
1.3) and (\cite{Sun}, 2.8.1).

The following definition comes from (\cite{Del2}, 6.2.4).

\begin{definition}\label{D3.1}
Let $K_0\in D_c(\s X_0,\overline{\bb Q}_{\ell})$
and $w\in\bb R.$ We say that $K_0$ \emph{has $\iota$-weights
$\le w$} if for each $i\in\bb Z,$ the punctual $\iota$-weights
of $\s H^iK_0$ are $\le i+w,$ and we denote by
$D_{\le w}(\s X_0,\overline{\bb Q}_{\ell})$ the full
subcategory of such complexes. We say that \emph{$K_0$ has
$\iota$-weights $\ge w$} if its Verdier dual $DK_0$ has
$\iota$-weights $\le-w,$ and denote by $D_{\ge w}(\s
X_0,\overline{\bb Q}_{\ell})$ the subcategory of such
complexes. We say that \emph{$K_0$ is $\iota$-pure of weight
$w$} if it belongs to both $D_{\le w}$ and $D_{\ge w}.$
\end{definition}

\begin{lemma}\label{L-pres}
Let $P:X_0\to\s X_0$ be a presentation, and $K_0\in
D_c(\s X_0,\overline{\bb Q}_{\ell}).$ Then $K_0$
is $\iota$-mixed of weights $\le w$ (resp. $\ge w$) if and
only if $P^*K_0$ (resp. $P^!K_0$) is so.
\end{lemma}

\begin{proof}
The two statements are dual to each other, so
it suffices to consider only the case where $K_0$ has
weights $\le w.$ The ``only if" part is obvious, and the
``if" part follows from (\cite{Sun}, 2.8) and the assumption
that $P$ is surjective.
\end{proof}

\begin{lemma}\label{5.3.2}\emph{(stack version of (\cite{BBD},
5.3.1, 5.3.2))}
(i) For $\s F_0\in\emph{Perv}_{\le w}(\s X_0)$ (resp. $\s F_0
\in\emph{Perv}_{\ge w}(\s X_0)$), all of its sub-quotients
are $\iota$-mixed of weights $\le w$ (resp. $\ge w$).

(ii) Let $j:\s U_0\hookrightarrow\s X_0$ be an
immersion of algebraic stacks. Then for any real number
$w,$ the intermediate extension $j_{!*}$ (\cite{LO3}, 6)
respects $\emph{Perv}_{\ge w}$ and $\emph{Perv}_{\le w}.$
In particular, if $\s F_0$ is an $\iota$-pure
perverse sheaf on $\s U_0,$ then $j_{!*}\s
F_0$ is $\iota$-pure of the same weight.
\end{lemma}

\begin{proof}
(i) Note that the variant for mixed perverse sheaves on stacks is given
in (\cite{LO3}, 9.3). Recall that for a morphism $u:\s F_0\to\s G_0$ of
perverse sheaves, with cone $K_0$ in $D^b_c(\s X_0,\overline
{\bb Q}_{\ell}),$ we have $\text{Ker}(u)=\leftexp{p}{\s
H}^{-1}K_0$ and $\text{Coker}(u)=\leftexp{p}{\s H}^0K_0.$
Let $P:X_0\to\s X_0$ be a presentation of relative dimension
$d.$ Since $P^*[d]$ commutes with $\leftexp{p}{\s H}^i,$
we see that $P^*[d]:\text{Perv}(\s X_0)\to\text{Perv}(X_0)$
is an exact functor. If $\s F_0'$ is a sub-object (resp.
quotient object) of $\s F_0,$ then $D\s F_0'$ is a quotient
object (resp. sub-object) of $D\s F_0,$ so by duality it
suffices to prove the $``\le w"$ part. This follows from
the exactness of $P^*[d]$ and the $\iota$-mixed variant
of (\cite{BBD}, 5.3.1).

(ii) For a closed immersion $i,$ we see that $i_*$ respects
$D_{\ge w}$ and $D_{\le w},$ so we may assume that $j$ is
an open immersion. We only need to consider the case for
$\text{Perv}_{\le w},$ since the case for $\text{Perv}_{\ge
w}$ follows from $j_{!*}D=Dj_{!*}.$

Let $P:X_0\to\s X_0$ be a presentation, and let the
following diagram be 2-Cartesian:
$$
\xymatrix@C=.8cm @R=.6cm{
U_0 \ar[r]^-{j'} \ar[d]_-{P'} & X_0 \ar[d]^-P \\
\s U_0 \ar[r]_-j & \s X_0.}
$$
For $\s F_0\in\text{Perv}_{\le w}(\s U_0),$ by
(\ref{L-pres}) it suffices to show that $P^*j_{!*}\s
F_0\in D_{\le w}(X_0,\overline{\bb Q}_{\ell}).$ Let
$d$ be the relative dimension of $P.$ By (\cite{LO3}, 6.2)
we have
$$
P^*j_{!*}\s F_0=(P^*(j_{!*}\s F_0)[d])[-d]
=j'_{!*}(P'^*\s F_0[d])[-d].
$$
Since $P'^*\s F_0\in D_{\le w},\ P'^*\s
F_0[d]\in D_{\le w+d},$ and by (\cite{BBD}, 5.3.2),
$j'_{!*}(P'^*\s F_0[d])\in\text{Perv}_{\le w+d},$
and by definition $P^*j_{!*}\s F_0\in D_{\le w}.$
\end{proof}

\begin{corollary}\label{5.3.4}\emph{(stack version of
(\cite{BBD}, 5.3.4))}
Every simple perverse sheaf $\s F_0$ on an algebraic
stack $\s X_0$ is $\iota$-pure.
\end{corollary}

\begin{proof}
By (\cite{LO3}, 8.2ii), $\s F_0\simeq j_{!*}L_0[d]$ for
some essentially smooth irreducible
substack $j:\s U_0\hookrightarrow\s X_0$ of dimension $d,$
and a simple lisse sheaf $L_0$ on $\s U_0,$ which is punctually
$\iota$-pure by (\ref{3.4.1}ii).
The result follows from (\ref{5.3.2}ii).
\end{proof}

\begin{theorem}\label{5.4.1}\emph{(stack version of
(\cite{BBD}, 5.4.1, 5.4.4))}
Let $K_0\in D_c^b(\s X_0,\overline{\bb Q}_{\ell}).$ Then $K_0$
has $\iota$-weights $\le w$ (resp. $\ge w$) if and only if
$\leftexp{p}{\s H}^iK_0$ has $\iota$-weights $\le
w+i$ (resp. $\ge w+i$), for each $i\in\bb Z.$ In
particular, $K_0$ is $\iota$-pure of weight $w$ if and
only if each $\leftexp{p}{\s H}^iK_0$ is
$\iota$-pure of weight $w+i.$
\end{theorem}

\begin{proof}
The case of $``\ge"$ follows from the case of $``\le"$ and
$\leftexp{p}{\s H}^i\circ D=D\circ\leftexp{p}{\s H}^{-i}.$
So we only need to show the case of $``\le".$

Let $P:X_0\to\s X_0$ be a presentation of relative
dimension $d.$ Then $K_0$ has $\iota$-weights $\le w$ if
and only if (\ref{L-pres}) $P^*K_0$ has $\iota$-weights
$\le w,$ if and only if (\cite{BBD}, 5.4.1) each
$\leftexp{p}{\s H}^i(P^*K_0)$ has $\iota$-weights
$\le w+i.$ We have $\leftexp{p}{\s H}^i(P^*K_0)=
\leftexp{p}{\s H}^i(P^*(K_0[-d])[d])=P^*
\leftexp{p}{\s H}^i(K_0[-d])[d]=P^*(\leftexp{p}
{\s H}^{i-d}K_0)[d],$ so $P^*(\leftexp{p}{\s
H}^{i-d}K_0),$ and hence $\leftexp{p}{\s
H}^{i-d}K_0,$ has $\iota$-weights $\le w+i-d.$
\end{proof}

The category $\text{Perv}(\s X_0)$ is artinian
and noetherian (\cite{LO3}, 8.2i). By the \textit{irreducible
constituents} of a perverse sheaf we mean its Jordan-H\"older
components.

\begin{definition}\label{D3.2}
Let $\s F_0$ be a perverse sheaf on an
algebraic stack $\s X_0,$ and let $\beta\in\bb{R/Z}.$
We say that \emph{$\s F_0$ has $\iota$-weights in $\beta,$}
if all irreducible constituents of $\s F_0,$ which are
$\iota$-pure by (\ref{5.3.4}), have $\iota$-weights in
the coset $\beta$ (in the sense of (\ref{D3.1})).
\end{definition}

Now we give the perverse sheaf version of (\ref{3.4.1}i,
ii). For (ii), the variant for mixed perverse sheaves
(which is the stack version of (\cite{BBD}, 5.3.5)) is
given in (\cite{LO3}, 9.2).

\begin{theorem}\label{5.3.5}
Let $\s F_0$ be a perverse sheaf on $\s X_0.$

(i) $\s F_0$ has a unique decomposition $\s
F_0=\bigoplus_{\beta\in\bb{R/Z}}\s F_0(\beta)$ into
perverse subsheaves, called the \emph{decomposition
according to the weights mod} $\bb Z,$ such that for
each $\beta,$ the $\iota$-weights of $\s F_0(\beta)$ are in
$\beta$ (in the sense of (\ref{D3.2})). This decomposition,
in which almost all the $\s F_0(\beta)$'s are zero, is functorial in
$\s F_0.$

(ii) If the $\iota$-weights of $\s F_0$ are integers
(\ref{D3.2}), then there exists a unique finite increasing
filtration $W$ of $\s F_0$ by perverse subsheaves, called
the \emph{weight filtration,} such that $\emph{Gr}^W_i\s
F_0$ is $\iota$-pure of weight $i,$ for each $i.$ Every morphism
between such perverse sheaves on $\s X_0$ is
strictly compatible with their weight filtrations.
\end{theorem}

\begin{proof}
(i) By descent theory (\cite{LO3}, 7.1) we reduce to the case
where $\s X_0=X_0$ is a scheme. One can further replace $X_0$ by an open
affine covering, and assume $X_0$ is separated.

\begin{sublemma}
Let $K_0$ and $L_0$ be $\iota$-pure complexes in
$D^b_c(X_0,\overline{\bb Q}_{\ell})$ of $\iota$-weights $w$
and $w',$ respectively, and assume $w-w'\notin\bb Z.$ Then
$Ext^n(K_0,L_0)=0$ for all $n.$
\end{sublemma}

\begin{proof}
By (\ref{5.4.1}), $\leftexp{p}{\tau}_{\le i}K_0$ and
$\leftexp{p}{\s H}^iK_0$ are $\iota$-pure, of $\iota$-weights
$w$ and $w+i$ respectively. Since $RHom(K_0,L_0)$ is a
triangulated functor in both $K_0$ and $L_0,$ we may
assume they are both perverse sheaves, and hence simple
perverse sheaves (\ref{5.3.2}i).

To show $Ext^n(K_0,L_0)=0,$ it suffices to show, by (\cite
{BBD}, 5.1.2.5), that the $\iota$-weights of $Ext^n(K,L)$
are not integers, for all $n.$ Therefore, we may make a finite
extension of the base field $\bb F_q.$ By (\cite{BBD}, 4.3.1ii)
we have $K_0=j_{!*}F_0[d]$ (resp. $L_0=i_{!*}G_0[d']$) for
some irreducible smooth subscheme (since we can take a finite
base extension) $j:U_0\hookrightarrow X_0$ of dimension $d$ (resp.
$i:V_0\hookrightarrow X_0$ of dimension $d'$), and for some
irreducible lisse sheaf $F_0$ on $U_0$ (resp. $G_0$ on
$V_0$). The sheaf $F_0=j^*K_0[-d]$ is $\iota$-mixed (or use Lafforgue's
result), hence punctually $\iota$-pure by (\cite{Del2}, 3.4.1ii), and
therefore $\iota$-pure by (\cite{Del2}, 6.2.5b). By (\cite
{BBD}, 5.3.2), the punctual $\iota$-weight of $F_0$ is
$w-d.$ By (\cite{Del2}, 1.3.6), there exists a number $b\in\overline
{\bb Q}_{\ell}^*$ such that the determinant of the lisse
sheaf $F_0^{(b)}$ deduced from $F_0$ by twist (\ref{torsion}) has
finite order, therefore, by Lafforgue's result, $F_0^{(b)}$ is
punctually pure of weight 0. The same is true for $G_0^{(b')}$
for some $b'\in\overline{\bb Q}_{\ell}^*.$ We see that
$w_q(\iota b)=d-w$ and $w_q(\iota b')=d'-w'.$

Then we have
$$
R\s Hom(K_0,L_0)=R\s Hom(j_{!*}F_0^{(b)}[d],i_{!*}G_0^{(b')}
[d'])^{(b/b')}
$$
and, by the projection formula,
$$
a_*R\s Hom(K_0,L_0)=\big(a_*R\s
Hom(j_{!*}F_0^{(b)}[d],i_{!*}G_0^{(b')}[d'])\big)^{(b/b')},
$$
where $a:X_0\to\text{Spec }\bb F_q$ is the structural morphism. By
(\cite{Del2}, 6.1.11), the complex $a_*R\s
Hom(j_{!*}F_0^{(b)}[d],i_{!*}G_0^{(b')}[d'])$ is mixed, hence
$\s H^n(a_*R\s Hom(K_0,L_0)),$ whose underlying vector space is
$Ext^n(K,L),$ does not have integer punctual $\iota$-weights.
\end{proof}

For every $\beta\in\bb{R/Z},$ we apply (\cite{BBD}, 5.3.6) to
$\text{Perv}(X_0),$ taking $S^+$ (resp. $S^-$) to be the set of
isomorphism classes of simple perverse sheaves
(hence $\iota$-pure) of $\iota$-weights not in $\beta$ (resp.
in $\beta$). Then we get a unique sub-object $\s F_0(\beta)$
with $\iota$-weights in $\beta$ (\ref{D3.2}), such that $\s
F_0/\s F_0(\beta)$ has $\iota$-weights not in $\beta,$ and
$\s F_0(\beta)$ is functorial in $\s F_0.$ This extension
splits since $Ext^1=0.$ By induction on length we get the
decomposition, which is unique and functorial.

(ii) As in (\cite{LO3}, 9.2), we may assume $\s X_0=X_0$
is a scheme. The proof in (\cite{BBD}, 5.3.5) still applies.
Namely, by (\ref{5.1.15}ii), if $\s F_0$ and $\s G_0$ are
$\iota$-pure simple perverse sheaves on $X_0,$ of
$\iota$-weights $f$ and $g$ respectively, with $f>g,$ then
$Ext^1(\s G_0,\s F_0)=0.$ Then apply (\cite{BBD}, 5.3.6) for each integer $i,$ by
taking $S^+$ (resp. $S^-$) to be the set of isomorphism
classes of $\iota$-pure simple perverse sheaves on $X_0$ of
$\iota$-weights $>i$ (resp. $\le i$).
\end{proof}

\begin{blank}\label{affine-stab}
Let $k$ be a field and let $\mathcal X$ be a $k$-algebraic
stack. We say that $\mathcal X$ has \textit{affine
stabilizers} if for every $x\in\mathcal X(\overline{k}),$
the group scheme $\text{Aut}_x$ is affine. Since being
affine is fpqc local on the base, we see that for any finite
field extension $k'/k$ and any $x\in\mathcal X(k'),$ the group
scheme $\text{Aut}_x$ over $k'$ is affine.
\end{blank}

\begin{proposition}\label{5.1.14}\emph{(stack version of
(\cite{BBD}, 5.1.14))}
(i) The Verdier dualizing functor $D_{\s X_0}$ interchanges $D_{\le w}$
and $D_{\ge-w}.$

(ii) For every morphism $f$ of $\bb F_q$-algebraic
stacks, $f^*$ respects $D_{\le w}$ and $f^!$ respects
$D_{\ge w}.$

(iii) For every morphism $f:\s X_0\to\s Y_0,$ where $\s
X_0$ is an $\bb F_q$-algebraic stack with affine stabilizers,
$f_!$ respects $D_{\le w}^-$ and $f_*$ respects
$D_{\ge w}^+.$

(iv) $\otimes$ takes $D_{\le w}^-\times D_{\le w'}^-$
into $D_{\le w+w'}^-.$

(v) $R\s Hom$ takes $D_{\le w}^-\times D_{\ge
w'}^+$ into $D_{\ge w'-w}^+.$
\end{proposition}

\begin{proof}
(i), (ii) and (iv) are clear, and (v) follows from (iv).
For (iii), if $\mathscr X_0$ has affine stabilizers,
so are all fibers $f^{-1}(y),$ for $y\in\s
Y_0(\bb F_{q^v}),$ and the claim for $f_!$ follows
from the spectral sequence
$$
H^i_c(f^{-1}(\overline{y}),\s H^jK)\Longrightarrow
H^{i+j}_c(f^{-1}(\overline{y}),K)
$$
and (\cite{Sun}, 1.4), and the claim for $f_*$ follows
from the claim for $f_!.$
\end{proof}

\begin{corollary}\label{5.1.15}\emph{(stack version of
(\cite{BBD}, 5.1.15))}
Let $\s X_0$ be an $\bb F_q$-algebraic stack with affine
stabilizers, with structural map $a:\s X_0\to\emph{Spec
}\bb F_q.$ Let $K_0$ (resp. $L_0$) be in $D_{\le
w}^-(\s X_0,\overline{\bb Q}_{\ell})$ (resp.
$D_{>w}^+(\s X_0,\overline{\bb Q}_{\ell})$) for
some real number $w.$ Then

(i) $a_*R\s Hom(K_0,L_0)$ is in $D_{>0}^+(\emph{Spec }\bb
F_q,\overline{\bb Q}_{\ell}).$

(ii) $Ext^i(K_0,L_0)=0$ for $i>0.$

If $L_0\in D_{\ge w}^+,$ then $a_*R\s Hom(K_0,L_0)$
is in $D_{\ge0}^+,$ and we have

(iii) $Ext^i(K,L)^F=0$ for $i>0.$ Here $F$ is the Frobenius
(\ref{frob}). In particular, for
$i>0,$ the canonical morphism $Ext^i(K_0,L_0)\to Ext^i(K,L)$ is zero.
\end{corollary}

The proof is the same as (\cite{BBD}, 5.1.15), using the
above stability result for algebraic stacks with affine stabilizers.

The following is the perverse sheaf version of
(\ref{3.4.1}iii).

\begin{theorem}\label{5.3.8}\emph{(stack version of
(\cite{BBD}, 5.3.8))}
Let $\s X_0$ be an $\bb F_q$-algebraic stack
with affine stabilizers. Then every $\iota$-pure
perverse sheaf $\s F_0$ on $\s X_0$ is
geometrically semi-simple (i.e. $\s F$ is
semi-simple), hence $\s F$ is a direct sum of
perverse sheaves of the form $j_{!*}L[d_{\s U}],$
for inclusions $j:\s U\hookrightarrow\s X$ of
$d_{\s U}$-dimensional irreducible smooth
substacks, and for irreducible lisse sheaves $L$ on $\s U.$
\end{theorem}

\begin{proof}
Let $\s F'$ be the sum in $\s F$ of simple
perverse subsheaves; it is a direct sum, and is the
largest semi-simple perverse subsheaf of $\s F.$
Then $\s F'$ is stable under Frobenius, hence
descends to a perverse subsheaf $\s F_0'\subset
\s F_0$ ((\cite{BBD}, 5.1.2) holds for stacks also).
Let $\s F_0''=\s F_0/\s F_0'.$ By (\ref{5.1.15}iii), the extension
$$
\xymatrix@C=.5cm{
0 \ar[r] & \s F' \ar[r] & \s F \ar[r] & \s F'' \ar[r] & 0}
$$
splits, because $\s F_0'$ and $\s F_0''$ have
the same weight (\cite{LO3}, 9.3). Then $\s F''$
must be zero, since otherwise it contains a simple
perverse subsheaf, and this contradicts the maximality of
$\s F'.$ Therefore $\s F=\s F'$ is semi-simple. The other
claim follows from (\cite{LO3}, 8.2ii): we may replace
$\s U$ by $\s U_{\text{red}}$ and hence assume that it is smooth.
\end{proof}

\begin{theorem}\label{5.4.5}\emph{(stack version of
(\cite{BBD}, 5.4.5))}
Let $\s X_0$ be an $\bb F_q$-algebraic stack
with affine stabilizers, and let $K_0\in D_c^b(\s X_0,\overline{\bb
Q}_{\ell})$ be an $\iota$-pure complex. Then $K$ on $\s X$ is
isomorphic non-canonically to the direct sum of the shifted perverse
cohomology sheaves $(\leftexp{p}{\s H}^iK)[-i].$
\end{theorem}

\begin{proof}
By (\ref{5.4.1}), both $\leftexp{p}{\tau}_{<i}K_0$ and
$(\leftexp{p}{\s H}^iK_0)[-i]$ are $\iota$-pure of
the same weight as that of $K_0.$ Therefore, by
(\ref{5.1.15}iii), the exact triangle
$$
\xymatrix@C=.5cm{
\leftexp{p}{\tau}_{<i}K_0 \ar[r] & \leftexp{p}{\tau}_{\le
i}K_0 \ar[r] & (\leftexp{p}{\s H}^iK_0)[-i] \ar[r]
&}
$$
geometrically splits, i.e. we have
$$
\leftexp{p}{\tau}_{\le i}K\simeq\leftexp{p}{\tau}_{<i}K
\oplus(\leftexp{p}{\s H}^iK)[-i],
$$
and the result follows by induction.
\end{proof}

\begin{blank}
\textit{Proof of theorem (\ref{main-thm}).} For the second claim,
we may assume that $K_0$ is an irreducible perverse sheaf, hence is $\iota$-pure (\ref{5.3.4}), therefore it follows from the first one.
For the first claim, by (\cite{Ols1}, 5.17) and (\ref{5.1.14} iii), we see that $f_*K_0$ is $\iota$-pure, hence by (\ref{5.4.5}) we have
$$
f_*K\simeq\bigoplus_{i\in\bb Z}(\leftexp{p}{R}^if_*K)[-i].
$$
By (\ref{5.4.1}), each $\leftexp{p}{R}^if_*K_0$ is also $\iota$-pure,
therefore it is geometrically semi-simple by (\ref{5.3.8}).
\end{blank}

\begin{remark}
In order to generalize the decomposition theorem to algebraic
stacks over $\bb C,$ one needs some foundational results, such as
the generic base change theorem for $f_*$ on stacks, the comparison
between the adic derived category and the topological derived
category of a $\bb C$-algebraic stack, and so on. We only give the
statement of the decomposition theorem for $\bb C$-stacks here,
and we will publish the details of the proof somewhere else.
\end{remark}

\begin{theorem}\label{6.2.5}\emph{(stack version of
(\cite{BBD}, 6.2.5))}
Let $f:\mathcal X\to\mathcal Y$ be a proper morphism of
finite diagonal between $\bb C$-algebraic stacks with
affine stabilizers. If $K\in\s D_c^b(\mathcal X^{\emph{an}},
\bb C)$ is semi-simple of geometric origin, then $f^{\emph{an}}_*K$
is also bounded, and is
semi-simple of geometric origin on $\mathcal Y^{\text{an}}.$
\end{theorem}

\end{document}